\documentclass[a4paper,10pt]{amsart}
\usepackage[utf8]{inputenc}
\usepackage{amssymb,amsmath,amsthm}
\usepackage{mathrsfs}
\usepackage{graphicx}
\usepackage{pgf,tikz}
\usetikzlibrary{arrows}

\newcommand{\TT}{\mathbb{T}}
\newcommand{\PP}{\mathbb{P}}
\newcommand{\NN}{\mathbb{N}}

\newcommand{\xx}{\mathbf{x}}

\newcommand{\reg}{\operatorname{reg}}
\newcommand{\set}[1]{\left \{ #1 \right \}}
\newcommand{\rt}{\longrightarrow}

\newtheorem{theorem}{Theorem}[section]
\newtheorem{lemma}[theorem]{Lemma}

\newtheorem{prop}[theorem]{Proposition}
\newtheorem*{claim*}{Claim}

\theoremstyle{definition}

\begin{document}

\title[Regularity of the vanishing ideal over a parallel composition
of paths]{Regularity of the vanishing ideal over a \\parallel composition
of paths}
\author[A.~Macchia]{Antonio Macchia}
\author[J.~Neves]{Jorge Neves}
\address{CMUC, Department of Mathematics, University of Coimbra,
Apartado 3008, EC Santa Cruz
3001--501 Coimbra, Portugal.
}
\email{macchia.antonello@gmail.com}
\email{neves@mat.uc.pt}

\author[M.~Vaz Pinto]{\\Maria Vaz Pinto}
\address{
Departamento de Matem\'atica\\
Instituto Superior T\'ecnico\\
Universidade de Lisboa\\
Avenida Rovisco Pais, 1\\
1049-001 Lisboa, Portugal
}
\email{vazpinto@math.ist.utl.pt}

\author[R.~H.~Villarreal]{Rafael H.~Villarreal}
\address{
Departamento de
Matem\'aticas\\
Centro de Investigaci\'on y de Estudios
Avanzados del
IPN\\
Apartado Postal
14--740 \\
07000 Mexico City, D.F.
}
\email{vila@math.cinvestav.mx}

\thanks{This work was partially supported by the Centre for Mathematics of the
University of Coimbra -- UID/MAT/00324/2013, funded by the Portuguese
Government through FCT/MCTES and co-funded by the European Regional
Development Fund through the Partnership Agreement PT2020. The third author was partially supported by the
Center for Mathematical Analysis, Geometry and Dynamical Systems
of Instituto Superior Técnico, Universidade de Lisboa,
funded by FCT/Portugal through UID/MAT/04459/2013.}

\begin{abstract}
Let $G$ be a graph obtained by taking $r\geq 2$ paths and identifying 
all first vertices and identifying all the last vertices.
We compute the Castelnuovo--Mumford regularity of the quotient $S/I(X)$, where 
$S$ is the polynomial ring on the edges of $G$ and $I(X)$ is the vanishing ideal of 
the projective toric subset parameterized by $G$. The case we consider is the
first case where the regularity was unknown, following earlier computations (by several authors) 
of the regularity when $G$ is a tree, cycle, complete graph or 
complete bipartite graph, but specially in light of the
reduction of the computation of the regularity in the bipartite case 
to the computation of the regularity of 
the blocks of $G$. We also prove new inequalities relating the Castelnuovo--Mumford regularity 
of $S/I(X)$ with the combinatorial structure of $G$, for a general graph.  
\end{abstract}



\maketitle

\section{Introduction}
\label{sec: intro}

Let $K$ be a field and denote by $S$ the polynomial ring $K[t_1,\dots,t_s]$ with the standard grading.
If $M$ is a finitely generated graded $S$-module and
\begin{equation}\label{eq: minimal graded free resolution}
0\to F_c\stackrel{\phi_c}{\rt} \cdots {\rt} F_2 \stackrel{\phi_2}{\rt} F_1 \stackrel{\phi_1}{\rt} F_0 \rt M
\end{equation}
is a minimal graded free resolution, the Castelnuovo--Mumford regularity of $M$ is the integer:
$$
\reg M = \max_{i,j} \set{j-i \mid b_{ij} \neq 0},
$$
where $b_{ij}$ are the graded Betti numbers of $M$, defined by $F_i \cong \oplus_{j\in \mathbb{Z}} S(-j)^{b_{ij}}$.
The regularity of $M$ reflects the size of the degrees of the entries of the matrices in (\ref{eq: minimal graded free resolution}),
and therefore, in a certain sense, the complexity of $M$ as a graded module. In the case when $M=S/I$, with $I$ a
Cohen--Macaulay homogeneous ideal, we know that (cf.~\cite[Proposition~4.2.3]{monalg}):
\begin{equation}\label{eq: relation of reg with Hilbert Series}
\reg S/I= \max_j \set{j-c\mid b_{cj}\not = 0} = \deg F_{S/I}(t) + \dim S/I,
\end{equation}
where $F_{S/I}(t)$ is the Hilbert Series of the module $S/I$ in rational function form.
\smallskip

Recently, many authors have studied the Castelnuovo-Mumford regularity of ideals associated to some combinatorial structure. 
For square free monomial ideals generated in degree $2$, so-called \emph{edge ideals} as their
generators correspond to the edges of a graph (cf.~\cite[Chapter~6]{monalg}),
the regularity can be bounded using the induced matching number of the associated graph 
(cf.~\cite{HaVT08}, \cite[Lemma 2.2]{Ka06} and \cite{Wo14}). 
For chordal graphs, it has been shown that the regularity actually coincides with this graph invariant 
(see \cite[Corollary 6.9]{HaVT08}). Several families of binomial ideals associated
with a combinatorial structure have also been studied. The class of \emph{toric ideals},
i.e., the ideal of relations of the edge subring of a graph, whose generators correspond to 
\emph{even closed walks} on the graph (cf.~\cite[Chapter~8]{monalg}), is one such example. 
For a complete graph $\mathcal{K}_n$, the regularity of its edge subring is equal to $\lfloor n/2 \rfloor$, 
while for a complete bipartite graph $\mathcal{K}_{a,b}$, this invariant coincides with $\min \{a,b\}-1$ (cf.~\cite{monalg}).
Lower and upper bounds for the regularity of 
toric ideals, in terms of the structure of the underlying graph, have recently been established (cf.~\cite{BiO'KVT}).
Another class of binomial ideals which has been extensively studied in recent times is the class of \emph{binomial edge ideals}. 
These ideals are generated by the maximal minors of a $2 \times s$ 
generic matrix, whose column indices correspond to the edges of a graph. The regularity of these ideals 
can also be expressed and bounded in terms of graph-theoretic invariants (cf.~\cite{EnZa15,KiSM16,MaMu13}).
\smallskip

For the purposes of this work, $K$ will be a finite field of cardinality $q$.
In the rest of the paper all \emph{graphs} will be undirected and without loops; multiple edges
are allowed. The vertex set of a graph $G$ will be denoted by $V_G$ and its edge set by $E_G$.
We denote the number of edges by $s$ and we fix an ordering of the set of edges given by
an identification of $E_G$ with the set of variables
of $K[t_1,\dots,t_s]$. If $H$ is a subgraph of $G$ we denote by $K[E_H]$ the polynomial subring on the variables of
$E_H$, under the above identification. To $G$ we associate a set $X$ defined by
\begin{equation}\label{eq: definition of X}
X = \set{ (\xx^{t_1},\xx^{t_2},\dots,\xx^{t_s})\in \PP^{s-1} \mid \xx \in (K^*)^{V_G}},
\end{equation}
where, if $\xx=\sum_{v\in V_G} x_v v$, with $x_v\in K^*$, for all $v\in V_G$, and $t_i$ is the edge $\set{v,w}$ (with $v\not = w$),
we set $\xx^{t_i}=x_{v}x_{w}$.
As $\xx^{t_i}\not = 0$, for all $i$, $X$ is a subset of the projective torus $\TT^{s-1}\subset \PP^{s-1}$.
We refer to $X$ as \emph{the projective toric subset parameterized by $G$}. Denote by $I(X)$
the vanishing ideal of $X$. Observe that 
$$I(\mathbb{T}^{s-1})=(t_1^{q-1}-t_s^{q-1},\dots,t_{s-1}^{q-1}-t_s^{q-1}) \subset I(X).$$

\smallskip

The notion of parameterized projective toric subsets and the study of their vanishing ideals
was introduced in \cite{ReSiVi11}. Unlike in the case of the edge ideal of $G$, we know that
$I(X)$ is always a Cohen--Macaulay homogeneous binomial ideal of height $s-1$ (Cf.~\cite[Theorem~2.1]{ReSiVi11}).
\smallskip

In the original definition of a parameterized projective toric subset, $G$ is assumed to be a simple graph.
However, on the one hand, we note that multiple edges play no part in the invariants of $K[E_G]/I(X)$. More precisely, if
$G'$ is the simple graph obtained from $G$ by removing all extra edges through any two given vertices
and $X'$ is the projective toric subset parameterized by $G'$, then $$K[E_{G'}]/I(X') \cong K[E_{G}]/I(X),$$
simply because $t_{j}-t_i\in I(X)$, for every extra edge $t_j$ between the endpoints of $t_i$.
On the other hand, allowing extra edges eases notation and simplifies statements and proofs.
\smallskip

As $X$ is a finite set, the value of the Hilbert function of $K[E_G]/I(X)$ is eventually equal
to $|X|$, the cardinality of $X$; therefore, $\deg K[E_G]/I(X)=|X|$.
A formula for the degree was first given in \cite{ReSiVi11} for connected graphs and then generalized to any graph in \cite[Theorem~3.2]{NeVPVi15}:
\begin{equation}\label{eq: degree}
\deg K[E_G]/I(X) =
\begin{cases}
\left (\frac{1}{2}\right)^{\gamma -1 } (q-1)^{n-m+\gamma-1}, \text{ if } \gamma\geq 1\text{ and } q\text{ is odd,}\\
(q-1)^{n-m+\gamma-1},\text{ if } \gamma\geq 1\text{ and } q\text{ is even,}\\
(q-1)^{n-m-1}, \text{ if }\gamma=0,
\end{cases}
\end{equation}
where ($q$ is the cardinality of $K$), $n$ is the cardinality of $V_G$,
$m$ is the number of connected components of $G$ and $\gamma$ the number of those that are non-bipartite.
\smallskip

Using the identity (\ref{eq: relation of reg with Hilbert Series}) and the fact that $\dim K[E_G]/I(X)=1$, we deduce that
the regula\-rity of $K[E_G]/I(X)$ coincides with its \emph{regularity index}, i.e., the minimum degree $d$ for which
the value of the Hilbert function at $k$ is equal to the value of the Hilbert polynomial at $k$, for every \mbox{$k\geq d$.}
(Cf.~\cite[Corollary~4.1.12]{monalg}.)
Since the Hilbert function of $K[E_G]/I(X)$ is strictly increa\-sing for \mbox{$0\leq d\leq \reg K[E_G]/I(X)$}
and the Hilbert polynomial is equal to $|X|=\deg K[E_G]/I(X)$ we conclude  that
$\reg K[E_G]/I(X)$ is the minimum $d$ for which the value of the Hilbert function at $d$ is
equal to $|X|=\deg K[E_G]/I(X)$.
\smallskip

In Table~\ref{table: values of reg}
we list cases for which this invariant is known.
\begin{table}[h]
\renewcommand{\arraystretch}{1.5}
\begin{tabular}{l|l}
 & $\reg K[E_G]/I(X)$ \\
\hline
$X=\TT^{s-1}$ & $(s-1)(q-2)$ \\
\hline
$G=\mathcal{K}_n$ & $\lceil (n-1)(q-2)/2 \rceil$ \\
\hline
$G=\mathcal{K}_{a,b}$ & $(\max\set{a,b}-1)(q-2)$ \\
\hline
$G=C_{2k}$ & $(k-1)(q-2)$ \\
\hline
$G=\mathcal{K}_{\alpha_1,\dots,\alpha_r}$ & $\max\set{\alpha_1(q-2),\dots,\alpha_r(q-2),\lceil (n-1)(q-2)/2\rceil}$ \\
\hline
\end{tabular}
\bigskip

\label{table: values of reg}
\caption{Known values of $\reg K[E_G]/I(X)$}
\end{table}
When $X$ coincides with the projective torus $\TT^{s-1}$ (which, from (\ref{eq: degree}), is the case, for example, if $G$ is a tree or an odd cycle),
$$
I(X) = (t_1^{q-1}-t_s^{q-1},\dots,t_{s-1}^{q-1}-t_s^{q-1}).
$$
Thus the regularity can be computed from (\ref{eq: relation of reg with Hilbert Series}), (see also \cite{SaVPVi11}).
The regularity in the case \mbox{$G=\mathcal{K}_n$} is given in \cite[Remark~3]{GoReSa13}.
The case $G= \mathcal{K}_{a,b}$ is given in \cite[Corollary~5.4]{GoRe08} and the case of
an even cycle, $G=C_{2k}$, in \cite[Theorem~6.2]{NeVPVi15}. In the case of a complete multipartite graph,
\mbox{$G=\mathcal{K}_{\alpha_1,\dots,\alpha_r}$} this invariant was computed in \cite[Theorem~4.3]{NeVP14}.
(Here $r \geq 3$ and the $n$ in the formula is \mbox{$|V_G|=\alpha_1+\cdots +\alpha_r$.})
\smallskip	

A graph $G$ is said to be $2$-connected if $|V_G|>2$ and, for every vertex $v\in V_G$, the graph
$G-v$ is connected. Any graph decomposes into \emph{blocks}, which consist of either maximal $2$-connected subgraphs,
single edges or isolated vertices. When $G$ is bipartite, we know that $\reg(E_G)/I(X)$ can be computed
from its block decomposition. More precisely, if $G$ is a simple bipartite graph
with no isolated vertices and $H_1,\dots,H_r$ are the blocks of $G$, then
\begin{equation}\label{eq: reg additive on blocks}
\reg K[E_G]/I(X) = \sum_{k=1}^r \reg K[E_{H_k}]/I(X_k) + (r-1)(q-2),
\end{equation}
where $X_k$ is the projective toric subset parameterized by the graph $H_k$, for each $k=1,\dots,r$
(cf.~\cite[Theorem~7.4]{NeVPVi14}).
This reduces the problem of computing
$\reg K[E_G]/I(X)$ for a bipartite graph to the case of $2$-connected graphs.
Notice that (\ref{eq: reg additive on blocks}), together with the formula for the regularity
in the case of even cycles, gives the regularity for any bipartite \emph{cactus graph} (a simple graph the blocks 
of which are edges or even cycles).
\smallskip

A $2$-connected graph can be reconstructed from one
of its cycles by adding a path by its endpoints (also known as an \emph{ear}) to the cycle and successively
repeating this operation (a finite number of times) to the graphs obtained (cf.~\cite[Proposition~3.1.1]{Di10}).
The simplest $2$-connected graph is a cycle. The second simplest $2$-connected graph is a cycle with an attached ear.
This graph can also be obtained by identifying the endpoints of $3$ paths, which, in turn, is also
known as the \emph{parallel composition}
of $3$ paths. Therefore the parallel composition of $3$ paths is
the first case of a $2$-connected graph for which the regularity of $K[E_G]/I(X)$ was
not known.
\smallskip

The aim of this work
is to compute the Castelnuovo--Mumford regularity of $K[E_G]/I(X)$, when $X$ belongs to the
family of projective toric subsets parameterized by a graph given as the
parallel composition of $r\geq 2$ paths, as illustrated in Figure~\ref{fig: G v2}.
(Notice that this graph may well have multiple edges if more than one $P_i$ has length equal to $1$.)

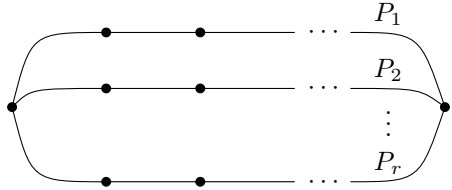
\begin{figure}[h]
\begin{center}
\begin{tikzpicture}[line cap=round,line join=round, scale=1.65]


\draw [fill=black] (0,-.6) circle (1pt);
\draw [fill=black] (3.45,-.6) circle (1pt);

\draw (0,-.6)..controls (0.15,0)..(.75,0);
\draw [fill=black] \foreach \x in {0.75,1.5}
    {(\x,0)  circle (1pt) -- (\x+.75,0)  };
\draw (2.5,0) node{$\cdots$};
\draw (2.7,0)..controls (3.25,0)..(3.45,-.6);
\draw (3,.15) node {$P_1$};

\draw (0,-.6)..controls (0.15,-.45)..(.75,-.45);
\draw [fill=black] \foreach \x in {0.75,1.5}
    {(\x,-.45)  circle (1pt) -- (\x+.75,-.45) };
\draw (2.5,-.45) node{$\cdots$};
\draw (2.7,-.45)..controls (3.25,-.45)..(3.45,-.6);
\draw (3,-.3) node {$P_2$};

\draw (0,-.6)..controls (0.15,-1.2)..(.75,-1.2);
\draw [fill=black] \foreach \x in {0.75,1.5}
    {(\x,-1.2)  circle (1pt) -- (\x+.75,-1.2) };
\draw (2.5,-1.2) node{$\cdots$};
\draw (2.7,-1.2)..controls (3.25,-1.2)..(3.45,-.6);
\draw (3,-.67) node{$\vdots$};
\draw (3,-1.05) node {$P_r$};

\end{tikzpicture}
\end{center}
\caption{$G$, the parallel composition of paths $P_1,P_2,\dots,P_r$.}
\label{fig: G v2}
\end{figure}

\smallskip
\noindent

Our first main result concerns the bipartite case.
\begin{theorem}\label{th: main theorem}
Let $X$ be the projective toric subset parameterized by the parallel composition
of $r\geq 2$ paths, the lengths of which, $k_1,\dots,k_r$, have the same parity. Then
$$
\reg K[E_G]/I(X) =
\begin{cases}
 (\lfloor k_1/2\rfloor +\cdots +\lfloor k_r/2 \rfloor)(q-2),\;  \text{if $k_i$ are odd,}\\
 ( k_1/2 +\cdots +k_r/2 -1)(q-2),\;  \text{if $k_i$ are even.}
\end{cases}
$$
\end{theorem}

We prove this result in Section~\ref{sec: main thm}, by proving the two inequalities involved.
The lower bound is a straightforward consequence of the fact that $G$ is bipartite
(cf.~(\ref{eq: lower bound on reg for bipartite G}) and Lemma~\ref{lemma: easy inequality}, below).
For the upper bound we divide the proof into two cases. The case of $k_i$ even is worked out by
induction on $r$ and arguing using suitable coverings of $G$ (cf.~Proposition~\ref{prop: opp inequality in the case k_i even}).
The case of $k_i$ odd is harder and relies on a characterization of the homogeneous binomials in $I(X)$
(cf.~Theorem~\ref{thm: opp inequality in the case k_i odd}).

With Theorem~\ref{th: main theorem} we are able to study the non-bipartite case.

\begin{theorem}\label{th: main theorem 2}
Let $X$ be the projective toric subset parameterized by a graph $G$ that is the parallel composition of 
$r\geq 2$ paths, the lengths of which have mixed parities. Then 
$$
\reg K[E_G]/I(X) = \reg K[E_{H_1}]/I(X_1) + \reg K[E_{H_2}]/I(X_2) + (q-2),
$$
where $H_1$ is the parallel composition of the paths of odd lengths, 
$H_2$ is the parallel composition of the paths of even lengths, and $X_1, X_2$, 
respectively, are the projective toric subsets they parameterize.
\end{theorem}

We point out that the formula of Theorem~\ref{th: main theorem 2} includes the case when 
only one path has length of different parity. In this situation, the corresponding summand of the formula 
does not follow from Theorem~\ref{th: main theorem}, rather, it can be retrieved 
from the first formula of Table~\ref{table: values of reg};
more precisely, if $H_i$ consists of a path of length $k$ then $\reg K[E_{H_i}]/I(X_i) = (k-1)(q-2)$.
\smallskip

The proof of Theorem~\ref{th: main theorem 2} occupies the second half of Section~\ref{sec: main thm}. As with our other main result 
we prove the two inequalities separately (cf.~Lemma~\ref{lemma: easy inequality in the non-bipartite case} and 
Theorem~\ref{thm: harder inequality in the non-bipartite case}). This time, the easier inequality is the one giving the upper bound.
For the lower bound inequality we need to use different techniques to those used in the proof of Theorem~\ref{th: main theorem}.

Section~\ref{sec: prelim} provides the background theory and the results that are used in our proofs. 
We single out the new contributions of Proposition~\ref{prop: general lower bound in contraction}, 
Proposition~\ref{prop: regularity inequality over subgraphs} and Proposition~\ref{prop: lower bound using independent set}, as we believe 
these results will prove useful in the study of the regularity for a general graph.

\section{Preliminaries}
\label{sec: prelim}

Let $K$ be a finite field of cardinality $q$. As in Section~\ref{sec: intro}, $G$ will denote a graph with
edge set $E_G$ of cardinality $s$ (we always assume that $G$ has no isolated vertices). We
fix an identification of the variables of $K[t_1,\dots,t_s]$ with $E_G$ and denote the
former by $K[E_G]$. Let $X$ be the projective toric subset parameterized by $G$, as
defined in (\ref{eq: definition of X}). If $a=(a_1,\dots,a_s)\in \NN^s$, $t^a$ denotes
the monomial $t_1^{a_1}\cdots t_s^{a_s}\in K[E_G]$.
\smallskip

We start by recalling a criterion for membership in $I(X)$ of a homogeneous binomial
that only involves the combinatorics of $G$.
It involves checking a linear congruence at every vertex of the graph.
\begin{figure}[h]
\begin{center}
\begin{tikzpicture}[line cap=round,line join=round, scale=.5]

\draw(0,0) -- (1.5,2.5);
\draw(1.5,2.5) -- (2.25,3);
\draw(1.5,2.5) -- (2.25,2);


\draw(0,0) -- (-2.5,1.5);
\draw(-2.5,1.5) -- (-2,2.25);
\draw(-2.5,1.5) -- (-3,2.25);

\draw(0,0) -- (-3,-1);
\draw(-3,-1) -- (-3.5,-1.75);
\draw(-3,-1) -- (-3.8,-0.75);

\draw(0,0) -- (1.5,-2.5);
\draw(1.5,-2.5) -- (2.25,-2.75);
\draw(1.5,-2.5) -- (1.75,-3.25);

\draw(0,0) -- (3,-.5);
\draw(3,-.5) -- (3.75,0);


\draw [fill=white,thick] (0,0) circle (14pt);
\draw (0,0) node {$v$};

\draw [fill=white,thick] (1.5,2.5) circle (14pt);
\draw [fill=white,thick] (-2.5,1.5) circle (14pt);
\draw [fill=white,thick] (-3,-1) circle (14pt);
\draw [fill=white,thick] (1.5,-2.5) circle (14pt);
\draw [fill=white,thick] (3,-.5) circle (14pt);

\draw(2,.1) node {$t_{i_1}$};
\draw(.5,1.85) node {$t_{i_2}$};
\draw(-1.65,.6) node {$t_{i_3}$};
\draw(-1.5,-1.1) node {$t_{i_4}$};

\draw(-.45,-1.77) node {$.$};
\draw(-.2,-1.85) node {$.$};
\draw(0.07,-1.84) node {$.$};

\draw(1.4,-1.45) node {$t_{i_r}$};

\end{tikzpicture}
\end{center}
\caption{Congruence at vertex $v$}
\label{fig: congruence}
\end{figure}
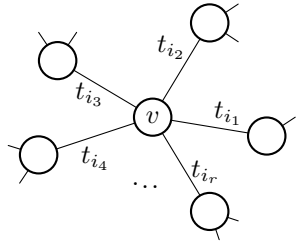
Let $v\in V_G$ and
let $t_{i_1},\dots,t_{i_r}$ be the edges incident to $v$ (cf.~Figure~\ref{fig: congruence}). Then by
\cite[Lemma~2.3]{NeVP14},
a homogeneous binomial $t^a-t^b\in K[E_G]$ belongs to $I(X)$ if and only if, for every vertex
$v\in V_G$, if $i_1,\dots,i_r$ are the indices of the edges incident to it, the congruence
\begin{equation}\label{eq: congruence}
a_{i_1}+\cdots +a_{i_r} \equiv b_{i_1}+\cdots + b_{i_r} \pmod{q-1}
\end{equation}
is satisfied. It follows easily from this criterion, that if
$H$ is a subgraph of $G$ and $Y$ is the projective toric subset parameterized by $H$, then 
$I(Y)= I(X)\cap K[E_H]$.

The following lemma will be
used in the proof of Theorem~\ref{thm: opp inequality in the case k_i odd}, below. Recall that
an ear of $G$ is a path which is maximal with respect to the condition that
all of its interior vertices have degree $2$ in $G$.

\begin{lemma}\label{lemma: swaping exponents}
Let $t^a-t^b\in K[E_G]$ be a homogeneous binomial. Let
$t_i$ and $t_j$ be edges along an ear of $G$ in a same parity position along this path.
Let $\sigma \colon K[E_G] \to K[E_G]$ be the automorphism
defined by swapping the two edges $t_i$ and $t_j$. Then
$$
t^a-t^b \in I(X) \iff \sigma(t^a) - \sigma(t^b) \in I(X).
$$
\end{lemma}

\begin{proof}
It is clear we can reduce to the case illustrated in Figure~\ref{fig: edges along ear}.
\begin{figure}[h]
\begin{center}
\begin{tikzpicture}[line cap=round,line join=round, scale=.5]

\draw (0,0) -- (9,0);

\draw (0,0) -- (-.9,-.5);
\draw (0,0) -- (-.8,.5);
\draw (-1,.2) node {$\vdots$};

\draw (9,0) -- (9.9,-.5);
\draw (9,0) -- (9.8,.5);
\draw (10,.2) node {$\vdots$};

\draw [fill=white,thick] \foreach \x in {0,3,6,9}
    {(\x,0) circle (14pt) };

\draw (0,0) node {$v_1$};
\draw (3,0) node {$v_2$};
\draw (6,0) node {$v_3$};
\draw (9,0) node {$v_4$};

\draw (1.5,.5) node {$t_{i}$};
\draw (4.5,.5) node {$t_{k}$};
\draw (7.5,.5) node {$t_{j}$};

\end{tikzpicture}
\end{center}
\caption{Swapping edges along an ear of $G$.}
\label{fig: edges along ear}
\end{figure}
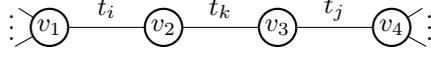
Since $\sigma(t^{{a}})-\sigma(t^{{b}})$ is homogeneous if and only if $t^a-t^b$ is, it suffices to check the equivalence of
the system of $4$ linear congruences given by the $4$ vertices $v_1,v_2,v_3$ and $v_4$.
Let ${E(v_i)}$ denote the set of edges incident to $v_i$ and denote
by $E_1$ the set $E(v_1)\setminus \set{t_i}$ and, likewise, $E_4=E(v_4)\setminus \set{t_j}$.
Let
\[
A_1 =\sum_{t_\ell\in E_1} a_\ell,\quad
A_4 =\sum_{t_\ell\in E_4} a_\ell,\quad
B_1 =\sum_{t_\ell\in E_1} b_\ell,\quad \text{and}\quad
B_4 =\sum_{t_\ell\in E_4} b_\ell.
\]
Then, we need to show that the two systems of congruences modulo $q-1$
\[
\begin{cases}
A_1+a_i \equiv B_1 + b_i \\
a_i+a_k \equiv b_i + b_k \\
a_k+a_j \equiv b_k + b_j \\
a_j+A_4 \equiv b_j + B_4 \\
\end{cases} \quad \text{and}\quad\:\:
\begin{cases}
A_1+a_j \equiv B_1 + b_j \\
a_j+a_k \equiv b_j + b_k \\
a_k+a_i \equiv b_k + b_i \\
a_i+A_4 \equiv b_i + B_4 \\
\end{cases}
\]
are equivalent, which is clearly true.
\end{proof}

Our approach to computing $\reg K[E_G]/I(X)$ is to consider an Artinian quotient
$K[E_G]/I(X,g)$, where $g\in K[E_G]$ is a suitable monomial.

\begin{prop}\label{prop: membership of a zero dimensional reduction}
Let $g\in K[E_G]$ be a monomial.
\begin{enumerate}
\item There exists a monomial order and
a binomial Gr\"obner basis $\mathcal B$ of $I(X)$ such that $\mathcal B \cup \set{g}$ is a Gr\"obner basis for the ideal $(I(X),g)\subset K[E_G]$.
\item A monomial $t^a\in K[E_G]$ belongs to $(I(X),g)$ if and only if there exists a monomial \mbox{$t^b\in K[E_G]$} such that $t^a-gt^b$ is homogeneous and belongs to $I(X)$.
\end{enumerate}
\end{prop}

\begin{proof}
Since $I(X)$ is generated by homogeneous binomials, the Gr\"obner basis obtained from such a set, 
after fixing any monomial order,
consists of homogeneous binomials,
by Buchberger's Algorithm. Let $t_{i_1},\dots,t_{i_r}$ be the variables dividing $g$. 
Fix the graded reverse lexicographical order after reordering the variables in way such that 
\mbox{$t_{i_1}\succ\cdots \succ t_{i_r}$} are the last variables of the ring.
Let $\mathcal{B}$ be a binomial Gr\"obner basis of $I(X)$
with respect to such order. To prove (i) it suffices to show that $S(f,g)$ reduces to $0$ modulo $\mathcal{B}\cup \set{g}$, for every
$f\in \mathcal{B}$. Let $f=t^a-t^b\in \mathcal{B}$. Assume, without loss of generality,
that $\operatorname{lt}(f)=t^a$. If $t_{i_r}$ divides $t^a$, then $t_{i_r}$ does not divide $t^b$ (we
may assume the generating set we start with consists of irreducible binomials). This implies that
$t^b\succ t^a$, hence $t_{i_r}$ does not divide $t^a$. Arguing in the same way, by induction,
we conclude that none of $t_{i_1},\dots,t_{i_r}$ divides $t^a$ and thus $\gcd(g,t^a)=1$.
Accordingly,
$$
S(f,g) = g(t^a-t^b) - t^ag = -gt^b
$$
which reduces to zero modulo $\mathcal{B}\cup \set{g}$. This completes the proof of (i).

Let $t^a$ be a monomial. One direction of the equivalence in (ii) is clear.
Assume that \mbox{$t^a\in (I(X),g)$.}	 Then, considering the Gr\"obner basis $\mathcal{B}\cup \set{g}$ obtained in (i),
$t^a$ has zero remainder after division with $\mathcal{B}\cup \set{g}$. Since
the division of a monomial by a binomial is still a monomial, the division algorithm stops the first time $g$ is used. Thus, the partial quotients of division are monomials
$t^a = h_0$, $h_1,\dots,h_k$
such that $h_{i}-h_{i-1}\in I(X)$, for all $i=1,\dots,k$ and such that $g$ divides $h_k$. Writing $h_k=gt^b$, we get a homogeneous binomial $t^a-gt^b$ which 
belongs to $I(X)$, as required.
\end{proof}

\begin{prop}\label{prop: computing reg by reducing to Artinian quotient}
Let $g\in K[E_G]$ be a monomial. Then $K[E_G]/(I(X),g)$
is zero in degree $d$ if and only if $d\geq \reg K[E_G]/I(X) + \deg(g)$.
\end{prop}

\begin{proof}
We denote $K[E_G]/I(X)$ by $R$ and, by abuse of notation, $K[E_G]/(I(X),g)$ by $R/g$.
Since $g$ is an $R$-regular element and $R$ is Cohen--Macaulay,
$$
\dim R/g = \dim R-1 = 0.
$$
Moreover, since $R/g$ is a quotient of a polynomial ring with the standard grading by a homogeneous ideal, its regularity
index is the minimum degree $d$ for which $(R/g)_d=0$.
(It is easy to see that $(R/g)_d=0$, for some $d$, implies $(R/g)_{d+k}=0$, for
all $k\geq 0$.) Hence we need to show that the regularity index of $R/g$ is equal to
$\reg K[E_G]/I(X) + \deg(g)$. Consider the following exact sequence of graded $K[E_G]$-modules:
\begin{equation}\nonumber
0\to  R[-\deg(g)]\stackrel{\cdot g}{\longrightarrow} R \to R/g\to 0.
\end{equation}
Comparing the degree of the Hilbert series of the three terms
and using the identity (\ref{eq: relation of reg with Hilbert Series}), we get $\deg F_{R/g} +1 = \reg R  +\deg(g)$, where $F_{R/g}$ is the Hilbert Series
of the $K[E_G]$-module $R/g$ in rational function form. As $\deg F_{R/g} +1$ is the
regularity index (cf. \cite[Corollary~4.1.12]{monalg}), we have proved the claim.
\end{proof}

We note that the following proposition can be easily derived from \cite[Theorem~7.4]{NeVPVi14}
in the bipartite case, and from \cite[Corollary~3.10]{SaVPVi11} and \cite[Lemma~1]{GoReHe03} 
in the non-bipartite case, when G is a unicyclic connected graph and the only cycle of G is odd.
Here, we do not assume G is bipartite nor a unicyclic connected graph with an odd cycle.

\begin{prop}\label{prop: aditivity on leaves}
Let $v\in V_G$ be a vertex of degree $1$. Assume that $|E_G|>1$.
Consider the graph \mbox{$G'=G-v$} and denote by $X'$ the projective toric subset parameterized by it. Then 
$$\reg K[E_G]/I(X) = \reg K[E_{G'}]/I(X') + (q-2).$$
\end{prop}

\begin{proof}
Let $t_i\in E_G$ be incident to $v$ and let $t_j\in E_G\setminus t_i$. 
According to Proposition~\ref{prop: computing reg by reducing to Artinian quotient}, 
to show that $$\reg K[E_G]/I(X) \leq \reg K[E_{G'}]/I(X') + (q-2)$$ it suffices to show that
for any monomial $t^a\in K[E_G]$ of degree $\reg K[E_{G'}]/I(X') + (q-2)+1$
we have $t^a \in (I(X),t_j)$. Let $t^a$ be such a monomial. If $a_i\geq q-1$ then 
writing $t^a = t^{a'}t_i^{q-1}$ for some $a'\in \NN^s$, we get:
$$
t^a = t^{a'}(t_i^{q-1}-t_j^{q-1}) + t^{a'}t_j^{q-1} \in (I(X),t_j).
$$
Assume now that $a_i<q-1$. Consider $a'\in \NN^s$, with $a'_i=0$, such that $t^a = t^{a'}t_i^{a_i}$.
Then $\deg t^{a'}=\deg t^a - a_i \geq \reg K[E_{G'}]/I(X')+1$, by our assumptions. As $t^{a'}$ belongs to 
$K[E_{G'}]$, using Proposition~\ref{prop: computing reg by reducing to Artinian quotient} we get
$t^{a'} \in (I(X'),t_j)\subset K[E_{G'}]$. As $G'$ is a subgraph of $G$ we have $I(X')\subset I(X)$ and therefore  
$t^{a'} \in (I(X),t_j)$. 
\smallskip

Using the same idea, let us now show that 
$$\reg K[E_{G'}]/I(X') \leq  \reg K[E_{G}]/I(X) - (q-2).$$
Let $t^a\in K[E_{G'}]$ be a monomial of degree $\reg K[E_{G}]/I(X) - (q-2)+1$.
Then $t^at_i^{q-2}$ belongs to $K[E_G]$ and has degree $\reg K[E_{G}]/I(X)+1$. We deduce that 
$t^a t_i^{q-2} \in (I(X),t_j)$. By Proposition~\ref{prop: membership of a zero dimensional reduction},
there exists a monomial $t^b\in K[E_G]$ such that $t^at_i^{q-2} - t_jt^b \in I(X)$. However the congruence
at vertex $v$ gives  $b_i= q-2 + k(q-1)$, for some $k\geq 0$. 
Let $b'\in \NN^s$ be such that $b'_i=0$ and $t^b = t^{b'}t_i^{b_i}$. Then: 
$$
t^at_i^{q-2} - t_jt^b \in I(X) \implies t^a-t_jt_i^{k(q-1)}t^{b'}\in I(X)\implies t^a - t_j^{1+k(q-1)}t^{b'} \in I(X).
$$
Since $t^a - t_j^{1+k(q-1)}t^{b'} \in K[E_{G'}]$ and $I(X')=I(X)\cap K[E_{G'}]$, we deduce that 
$t^a \in (I(X'),t_j)$.
\end{proof}

Let $G$ be a connected graph and a spanning subgraph of a bipartite graph $H$. Let $Y$ be the projective toric subset parameterized by
$H$. Then, by \cite[Lemma~2.13]{VPVi13}, if $|X|=|Y|$, it follows that $$\reg K[E_G]/I(X)\geq \reg K[E_H]/I(Y).$$
Hence if $G$ is a connected bipartite spanning subgraph of $\mathcal{K}_{a,b}$, by (\ref{eq: degree}) 
the assumption on the cardinality of the associated parameterized projective toric subsets holds and we obtain:
\begin{equation}\label{eq: lower bound on reg for bipartite G}
\reg K[E_G]/I(X)\geq (\max\set{a,b}-1)(q-2).
\end{equation}

In the remainder of this section we introduce three new inequalities involving 
$\reg K[E_G]/I(X)$. They will play an important role in the proofs of
Theorem~\ref{th: main theorem} and Theorem~\ref{th: main theorem 2}.

\begin{prop}\label{prop: general lower bound in contraction}
Let $v_1$ and $v_2$ be two vertices of $G$ such that 
$\set{v_1,v_2}$ is a non-edge of $G$. Let $G'$ be the graph obtained by 
identifying $v_1$ with $v_2$ and denote by $X'$ the projective toric subset parameterized by it. 
Then $\reg K[E_G]/I(X)\geq  \reg K[E_{G'}]/I(X')$.
\end{prop}

\begin{proof}
The edge sets of $G$ and $G'$ have the same cardinality. Morevoer,
there is an induced identification of the edges of $G'$ with the variables of the 
polynomial ring $K[t_1,\dots,t_s]$ under which  $K[E_G]=K[E_{G'}]$.
\begin{figure}[h]
\begin{center}
\begin{tikzpicture}[line cap=round,line join=round, scale=2.5]

\draw (-0.5,-.5) node {$G$};

\draw [fill=black] (0,.5) circle (1pt);
\draw [fill=black] (0,-.5) circle (1pt);
\draw [fill=black] (-.7,0) circle (1pt);
\draw (-.7,0) -- (-.75,.15);
\draw (-.7,0) -- (-.8,-.1);

\draw [fill=black] (.5,.4) circle (1pt);
\draw (.5,.4) -- (.55,.55);
\draw (.5,.4) -- (.65,.45);

\draw [fill=black] (.5,-.4) circle (1pt);
\draw (.5,-.4) -- (.55,-.55);
\draw (.5,-.4) -- (.65,-.45);

\draw [fill=black] (-.3,.75) circle (1pt);
\draw (-.3,.75) -- (-.45,.8);
\draw (-.3,.75) -- (-.2,.85);

\draw [fill=black] (-.15,-.7) circle (1pt);
\draw (-.15,-.7) -- (-.25,-.75);
\draw (-.15,-.7) -- (-.1,-.8);

\draw [black] (-.7,0).. controls (-.65,.15) and (-.05,.55) .. (0,.5);
\draw [black] (-.7,0).. controls (-.65,-.15) and (-.05,-.55) .. (0,-.5);
\draw [black] (0,.5).. controls (.1,.55) and (.45,.45) .. (.5,.4);
\draw [black] (0,-.5).. controls (.1,-.55) and (.45,-.45) .. (.5,-.4);
\draw [black] (.5,.4).. controls (.6,.4) and (.6,-.4) .. (.5,-.4);

\draw [black] (-.3,.75)--(0,.5);
\draw [black] (-.15,-.7)--(0,-.5);

\draw (0,.35) node {$v_1$};
\draw (0,-.35) node {$v_2$};

\draw[gray,<->] (0,.25)--(0,-.25);

\draw [gray] (.9,0).. controls (1,.07) and (1,-.07) ..(1.1,0); 
\draw [gray,->] (1.1,0).. controls (1.2,.07) and (1.2,-.05) ..(1.3,0);


\draw (2,-.5) node {$G'$};

\draw [fill=black] (2.5,0) circle (1pt);

\draw [fill=black] (1.8,0) circle (1pt);
\draw (1.8,0) -- (1.75,.15);
\draw (1.8,0) -- (1.7,-.1);

\draw [fill=black] (3,.4) circle (1pt);
\draw (3,.4) -- (3.05,.55);
\draw (3,.4) -- (3.15,.45);

\draw [fill=black] (3,-.4) circle (1pt);
\draw (3,-.4) -- (3.05,-.55);
\draw (3,-.4) -- (3.15,-.45);

\draw [fill=black] (2.27,.3) circle (1pt);
\draw (2.27,.3) -- (2.11,.36);
\draw (2.27,.3) -- (2.38,.4);

\draw [fill=black] (2.37,-.22) circle (1pt);
\draw (2.37,-.22) -- (2.25,-.24);
\draw (2.37,-.22) -- (2.44,-.33);

\draw [black] (1.8,0).. controls (1.85,.05) and (2.45,.05) .. (2.5,0);
\draw [black] (1.8,0).. controls (1.85,-.05) and (2.45,-.05) .. (2.5,0);

\draw [black] (2.5,0).. controls (2.6,.15) and (2.95,.45) .. (3,.4);
\draw [black] (2.5,0).. controls (2.6,-.15) and (2.95,-.45) .. (3,-.4);
\draw [black] (3,.4).. controls (3.1,.4) and (3.1,-.4) .. (3,-.4);

\draw [black] (2.27,.3)--(2.5,0);
\draw [black] (2.37,-.22)--(2.5,0);

\draw [gray, ->] (2.3,.75)..controls (2.75,.7) and (2.5,.3) .. (2.5,.1);

\draw (2,.75) node {$v_1=v_2$};

\end{tikzpicture}
\end{center}
\caption{The graph obtained by identifying two vertices of $G$.}
\label{fig: identifying a vertex}
\end{figure}
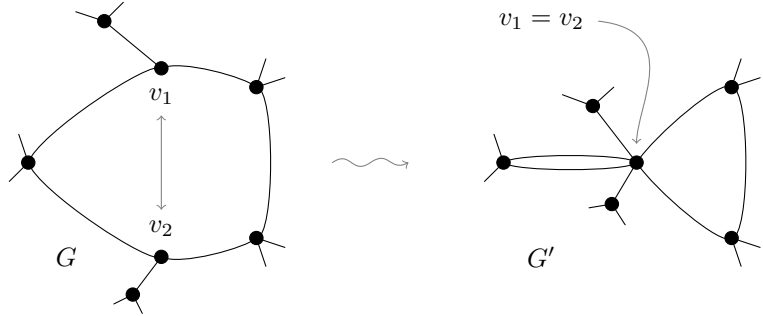
Since the parameterization of $X'$
is obtained by adding the restriction that the coefficient of $v_1$ in the formal sum 
$\sum_{v\in V_G} x_v v$ be equal to the coefficient of $v_2$ we obtain $X'\subset X$ (cf.~(\ref{eq: definition of X})), and thus, 
$I(X)\subset I(X')$. 
Let $t_1$ be an edge.
According to 
Proposition~\ref{prop: computing reg by reducing to Artinian quotient},
to show that $$\reg K[E_{G'}]/I(X') \leq \reg K[E_{G}]/I(X)$$ it suffices to prove 
that for any monomial $t^a$ of degree $\reg K[E_G]/I(X) +1$ we have $t^a \in (I(X'),t_1)$. Let 
$t^a$ be such a monomial. Then, using again Proposition~\ref{prop: computing reg by reducing to Artinian quotient}, 
we deduce that $(t^a\in I(X),t_1)$. Since $I(X)\subset I(X')$ we get $(t^a\in I(X'),t_1)$.
\end{proof}

\begin{prop}\label{prop: regularity inequality over subgraphs}
Let $H_1,H_2\subset G$ be subgraphs such that $E_G=E_{H_1}\cup E_{H_2}$ and $E_{H_1}\cap E_{H_2}\not = \emptyset$.
Let $X_1$ and $X_2$ be the projective toric subsets parameterized by $H_1$ and $H_2$ and $I(X_1)\subset K[E_{H_1}]$,
$I(X_2)\subset K[E_{H_2}]$ their corresponding vanishing ideals. Then
$$\reg K[E_G]/I(X)\leq \reg K[E_{H_1}]/I(X_1) + \reg K[E_{H_2}]/I(X_2).$$
\end{prop}

\begin{proof}
Let $t_i\in E_{H_1}\cap E_{H_2}$.
According to Proposition~\ref{prop: computing reg by reducing to Artinian quotient}, it
suffices to show that any monomial $t^a\in K[E_G]$, of degree $\reg K[E_{H_1}]/I(X_1) + \reg K[E_{H_2}]/I(X_2) + 1$,
belongs to $(I(X),t_i)$. Let us write $t^a=t^bt^c$ for some
$t^b\in K[E_{H_1}]$ and $t^c\in K[E_{H_2}]$. Since
\mbox{$\deg(t^a)=\deg(t^b)+\deg(t^c)$},
we have \mbox{$\deg(t^b)\geq \reg K[E_{H_1}]/I(X_1) +1$} or \mbox{$\deg(t^c)\geq \reg K[E_{H_2}]/I(X_2) + 1$.}
By Proposition~\ref{prop: computing reg by reducing to Artinian quotient} it follows that
$t^b \in (I(X_1),t_i)\subset K[E_{H_1}]$ or $t^c \in (I(X_2),t_i)\subset K[E_{H_2}]$, respectively.
In both cases we conclude that  $t^a\in (I(X),t_i)$.
\end{proof}

\begin{prop}\label{prop: lower bound using independent set}
Let $\set{v_1,\dots,v_r}$ be an independent set of vertices of $G$. Assume
that there is an edge in $G-\set{v_1,\dots,v_r}$. Then $\reg K[E_G]/I(X) \geq r(q-2)$. 
\end{prop}

\begin{proof}
By Proposition~\ref{prop: computing reg by reducing to Artinian quotient},
to show that $\reg K[E_G]/I(X) \geq r(q-2)$ it suffices to show that 
there exists and edge $t_i$ and a monomial $t^a\in K[E_G]$ of degree $r(q-2)$
that does not belong to $(I(X),t_i)$. Let $t_i$ be an edge of $G-\set{v_1,\dots,v_r}$ 
and, for every $i=1,\dots,r$, let $t_{j_i}$ be an edge incident to $v_i$. Such edges exist since we assume 
that $G$ has no isolated vertices. Notice also that since $\set{v_1,\dots,v_r}$ is an independent set the edges
$t_{j_1},\dots,t_{j_r}$ are distinct. Consider the monomial:
$$
t^a = (t_{j_1}\cdots t_{j_r})^{q-2}
$$
and let us show that $t^a\not \in (I(X),t_i)$. Suppose the contrary holds. 
Then, by Proposition~\ref{prop: membership of a zero dimensional reduction},
there exists a monomial $t^b$ such that $t^a-t_it^b$ is homogeneous and belongs to $I(X)$.
Since $t_i$ is not incident to any of the vertices of 
$\set{v_1,\dots,v_r}$, evaluating the congruence at a particular vertex of this set, 
we conclude that the degree of $t^b$ in the edges incident to it is $\geq q-2$. 
Since, by assumption, these vertices possess no common incident edges we deduce that the 
degree of $t^b$ in edges incident to the vertices of $\set{v_1,\dots,v_r}$ is $\geq r(q-2)$. In particular, 
$\deg(t^b)\geq r(q-2)$. But this implies that $t^a-t_it^b$ is not homogeneous, which is a contradiction. 
\end{proof}

We note that Proposition~\ref{prop: lower bound using independent set} implies (\ref{eq: lower bound on reg for bipartite G}).

\section{Proof of the main results}
\label{sec: main thm}

The aim of this section is to prove Theorem~\ref{th: main theorem} and Theorem~\ref{th: main theorem 2}.
In what follows $G$ is the parallel composition of $r\geq 2$ paths $P_1,\dots,P_r$ of lengths
$k_1,\dots,k_r$. In a first instance, we assume that these integers have the same parity, so that
$G$ is bipartite. If $r=2$ and one of $k_1,k_2$ is $>1$, then $G$ is an even cycle of length $k_1+k_2$.
In this case, by \cite[Theorem~6.2]{NeVPVi15}, we know that
$\reg K[E_G]/I(X) = ((k_1+k_2)/2-1)(q-2)$. If $r=2$ and $k_1=k_2=1$, then
$G$ is a graph on $2$ vertices with
$2$ multiple edges. Hence the value of the regularity is the same as in the case of a tree
with a single edge, which is $(s-1)(q-2)=0$ (cf.~Table~\ref{table: values of reg}).
Both cases agree with the formula in Theorem~\ref{th: main theorem}.

\begin{lemma}\label{lemma: easy inequality}
$$
\reg K[E_G]/I(X) \geq
\begin{cases}
 (\lfloor k_1/2\rfloor +\cdots +\lfloor k_r/2 \rfloor)(q-2),\; \text{if $k_i$ are odd,}\\
 ( k_1/2 +\cdots +k_r/2 -1)(q-2),\; \text{if $k_i$ are even.}
\end{cases}
$$
\end{lemma}

\begin{proof}
If $k_i$ are odd, then $G$ is a connected spanning subgraph of $\mathcal K_{\rho,\rho}$, where $\rho$
is the integer \mbox{$1 + \lfloor k_1/2\rfloor +\cdots +\lfloor k_r/2 \rfloor$.} If $k_i$
are even, then $G$ is a connected spanning subgraph of $\mathcal{K}_{(\rho-r+2),\rho}$ where
$\rho$ is the integer $k_1/2+\cdots +k_r/2$. Hence the claim follows from (\ref{eq: lower bound on reg for bipartite G}).
\end{proof}

In the next two results we prove the opposite inequalities in each case.
We need to fix some notation. For each $i\in \set{1,\dots,r}$, let $\sigma_i = k_1+\cdots + k_{i-1}$, so that, in particular,
$\sigma_1=0$.
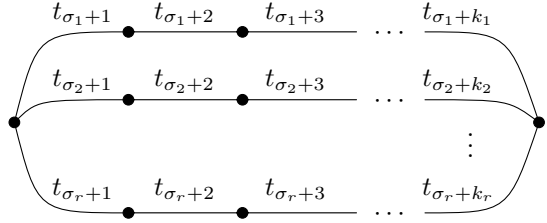
\begin{figure}[ht]
\begin{center}
\begin{tikzpicture}[line cap=round,line join=round, scale=2]


\draw [fill=black] (0,-.6) circle (1pt);
\draw [fill=black] (3.45,-.6) circle (1pt);

\draw (0,-.6)..controls (0.15,0)..(.75,0);
\draw [fill=black] \foreach \x in {0.75,1.5}
    {(\x,0)  circle (1pt) -- (\x+.75,0)  };
\draw (2.48,-.006) node{$\cdots$};
\draw (2.7,0)..controls (3.25,0)..(3.45,-.6);
\draw (.45,.12) node {$t_{\sigma_1+1}$};
\draw (1.12,.12) node {$t_{\sigma_1+2}$};
\draw (1.85,.12) node {$t_{\sigma_1+3}$};
\draw (2.92,.12) node {$t_{\sigma_1+k_1}$};

\draw (0,-.6)..controls (0.15,-.45)..(.75,-.45);
\draw [fill=black] \foreach \x in {0.75,1.5}
    {(\x,-.45)  circle (1pt) -- (\x+.75,-.45) };
\draw (2.48,-.456) node{$\cdots$};
\draw (2.7,-.45)..controls (3.25,-.45)..(3.45,-.6);
\draw (.45,.-.33) node {$t_{\sigma_2+1}$};
\draw (1.12,-.33) node {$t_{\sigma_2+2}$};
\draw (1.85,-.33) node {$t_{\sigma_2+3}$};
\draw (2.92,-.33) node {$t_{\sigma_2+k_2}$};

\draw (0,-.6)..controls (0.15,-1.2)..(.75,-1.2);
\draw [fill=black] \foreach \x in {0.75,1.5}
    {(\x,-1.2)  circle (1pt) -- (\x+.75,-1.2) };
\draw (2.48,-1.206) node{$\cdots$};
\draw (2.7,-1.2)..controls (3.25,-1.2)..(3.45,-.6);
\draw (3,-.7) node{$\vdots$};
\draw (.45,.-1.05) node {$t_{\sigma_r+1}$};
\draw (1.12,-1.05) node {$t_{\sigma_r+2}$};
\draw (1.85,-1.05) node {$t_{\sigma_r+3}$};
\draw (2.92,-1.05) node {$t_{\sigma_r+k_r}$};

\end{tikzpicture}
\end{center}
\caption{Labeling of the edges of $G$.}
\label{fig: labelling of edges}
\end{figure}

\noindent
Let us label the edges of $G$ as in Figure~\ref{fig: labelling of edges}. For each $i\in \set{1,\dots,r}$, let
$f_i,g_i\in K[E_G]$ be:
\begin{equation}\label{eq: def of f_i and g_i}
f_i=t_{\sigma_i+1}\cdot t_{\sigma_i+3} \cdots t_{\sigma_i+2\lceil k_i/2 \rceil-1}\quad \text{and} \quad
g_i=t_{\sigma_i+2}\cdot t_{\sigma_i+4} \cdots t_{\sigma_i+2\lfloor k_i/2 \rfloor}.
\end{equation}
(In other words, $f_i$ is the product of every other edge in $P_i$ starting with $t_{\sigma_i+1}$ and
$g_i$ is the product of every other edge in $P_i$ starting with $t_{\sigma_i+2}$.)
We notice that, for all $i\not = j$,
\begin{equation}
\label{eq: relation between fi and gj}
f_ig_j - f_jg_i \in  I(X).
\end{equation}

\begin{prop}\label{prop: opp inequality in the case k_i even}
If $k_i$ are even, then $\reg K[E_G]/I(X)\leq (k_1/2+\cdots +k_r/2 -1)(q-2)$.
\end{prop}

\begin{proof}
According to Proposition~\ref{prop: computing reg by reducing to Artinian quotient}, it suffices to show
that any monomial $t^a\in K[E_G]$ of degree \mbox{$(k_1/2+\cdots +k_r/2 -1)(q-2) +1$} belongs to $(I(X),t_1)$.
We may assume $t_1$ does not divide $t^a$. We will argue by induction on $r$. For $r=2$, as observed earlier, the result holds true.
Assume now that $r\geq 3$. Let $H$ be the subgraph of $G$ given by $\set{t_1}\cup P_2\cup \cdots \cup P_r$
and $Y$ be the projective toric subset parameterized by $G$.
By induction and \cite[Theorem~7.4]{NeVPVi14},
\[
\reg K[E_H]/I(Y) = (k_2/2+\cdots + k_r/2)(q-2).
\]
Set $t^a = t^bt^c$, with $t^b\in K[E_{P_1}]$ and $t^c\in K[E_{H}]$. If \mbox{$\deg(t^c)\geq (k_2/2+\cdots + k_r/2)(q-2)+1$},
then, by Proposition~\ref{prop: computing reg by reducing to Artinian quotient}, $t^c\in (I(Y),t_1)\subset (I(X),t_1)$, which implies that
$t^a \in (I(X),t_1)$.
Assume that \mbox{$\deg(t^c)\leq (k_2/2+\cdots + k_r/2)(q-2)$}. Then $\deg(t^b)\geq (k_1/2-1)(q-2)+1$.
Consider now the subgraphs of $G$ given by
$$
H_1=P_1\cup P_2 \quad \text{and}\quad H_2=P_1\cup P_3\cup \cdots \cup P_r
$$
and denote by $X_1$ and $X_2$, respectively, the projective toric subsets parameterized by them.
Set $t^c=t^dt^e$ with $t^bt^d\in K[E_{H_1}]$ and $t^bt^e\in K[E_{H_2}]$.
By the induction hypothesis,
$$
\renewcommand{\arraystretch}{1.3}
\begin{array}{c}
\reg K[E_{H_1}]/I(X_1) = (k_1/2+k_2/2-1)(q-2)\quad \text{and} \\
\reg K[E_{H_2}]/I(X_2) = (k_1/2+k_3/2+\cdots + k_r/2-1)(q-2).
\end{array}
$$
Hence, if $\deg(t^bt^d)\geq (k_1/2+k_2/2-1)(q-2) +1$, we get $t^bt^d \in (I(X_1),t_1)\subset (I(X),t_1)$ which
implies that $t^a\in (I(X),t_1)$. Similarly, if $\deg(t^bt^e)\geq (k_1/2+k_3/2+\cdots + k_r/2-1)(q-2) +1$.
Suppose that
$$
\renewcommand{\arraystretch}{1.3}
\begin{array}{c}
\deg(t^bt^d)\leq (k_1/2+k_2/2-1)(q-2)\quad \text{and} \\
\deg(t^bt^e)\leq (k_1/2+k_3/2+\cdots + k_r/2-1)(q-2).
\end{array}
$$
Since $\deg(t^a)=\deg(t^bt^d)+\deg(t^bt^e)-\deg(t^b)$,
we deduce that
$$
\renewcommand{\arraystretch}{1.3}
\begin{array}{c}
\deg(t^a) \leq (k_1/2+\cdots+k_r/2-1)(q-2)-1,
\end{array}
$$
which is a contradiction.
\end{proof}

\begin{theorem}\label{thm: opp inequality in the case k_i odd}
If $k_i$ are odd, then $\reg K[E_G]/I(X) \leq (\lfloor k_1/2\rfloor +\cdots+\lfloor k_r/2\rfloor)(q-2)$.
\end{theorem}

\begin{proof}
We will use induction on $k_1+\dots+k_r$.
In the base case, $r=2$ and $k_1=k_2=1$. As we mentioned earlier,
$\reg K[E_G]/I(X) = 0$.

Assume that $k_1+\cdots+k_r>3$ and, as induction hypothesis, that the statement
of the theorem holds for any $k'_1,\dots,k'_{r'}$ and $r'\geq 2$ such that
$k'_1+\cdots +k'_{r'}<k_1+\cdots+k_r$. If $r=2$, then, as observed in the
beginning of this section, $G$ is an even cycle of length $k_1+k_2$ and accordingly
$$\reg K[E_G]/I(X) = ((k_1+k_2)/2-1)(q-2)=(\lfloor k_1/2 \rfloor + \lfloor k_2/2 \rfloor)(q-2).$$
Hence, we may assume $r\geq 3$. If, for some $i$, $k_i=1$, let $G'$ be the subgraph of $G$ given as the parallel
composition of all $P_1,\dots,P_r$ but $P_i$. We note that $G'$ is a connected spanning subgraph of
$G$ and hence, if $X'$ is the projective toric subset parameterized by $G'$,
by the induction hypothesis, since $\lfloor k_i/2 \rfloor=0$, we get
$$\reg K[E_G]/I(X) \leq \reg K[E_{G'}]/I(X') \leq (\lfloor k_1/2\rfloor +\cdots+\lfloor k_r/2\rfloor)(q-2).$$
Thus, we may assume $k_i\geq 3$, for all \mbox{$i=1,\dots,r$.}
According to Proposition~\ref{prop: computing reg by reducing to Artinian quotient}, to show that
$\reg K[E_G]/I(X)\leq (\lfloor k_1/2\rfloor +\cdots+\lfloor k_r/2\rfloor)(q-2)$, it suffices to show that
any monomial $t^a\in K[E_G]$ of degree
\begin{equation}\label{eq: degree of t^a}
(\lfloor k_1/2\rfloor +\cdots+\lfloor k_r/2\rfloor)(q-2)+\lfloor k_1/2 \rfloor
+\cdots + \lfloor k_r/2 \rfloor
\end{equation}
belongs to the ideal $(I(X),g)\subset K[E_G]$,
where $g=g_1\cdots g_r$ and $g_i$ were defined in (\ref{eq: def of f_i and g_i}).
Let $t^a$ be one such monomial and write it as
the product of monomials, $h_1\cdots h_r$, where $h_i\in K[E_{P_i}]$.
By (\ref{eq: degree of t^a}), we have
$\deg(h_i)\leq \lfloor k_i/2 \rfloor (q-1)$, for some $i\in\set{1,\dots,r}$.
Without loss of generality we assume $i=1$. In particular,
\begin{equation}\label{eq: inequality4}
\deg(h_2\cdots h_r)\geq (\lfloor k_2/2\rfloor+\cdots +\lfloor k_r/2\rfloor)(q-2) + \lfloor k_2/2 \rfloor +\cdots + \lfloor k_r/2 \rfloor.
\end{equation}
Since $g$ is invariant under the swapping of variables corresponding to edges of $P_1$
in a same parity position,
using Lemma~\ref{lemma: swaping exponents}, we may assume that
\begin{equation}\label{eq: inequality5}
a_1\leq a_3 \leq \cdots \leq a_{2\lceil k_1/2\rceil -1}\text{ and }
a_2\leq a_4 \leq \cdots \leq a_{2\lfloor k_1/2\rfloor}.
\end{equation}
Let $H$ be the subgraph of $G$ given by $P_2\cup\cdots \cup P_r$ and denote by $Y$ the projective
toric subset parameterized by it. By induction,
$\reg K[E_H]/I(Y)=(\lfloor k_2/2 \rfloor +\cdots+\lfloor k_r/2 \rfloor)(q-2)$.
Then, by (\ref{eq: inequality4}),
Proposition~\ref{prop: computing reg by reducing to Artinian quotient}
and Proposition~\ref{prop: membership of a zero dimensional reduction},
there exists a monomial $t^b\in K[E_H]$, for some $b\in \NN^s$ supported
on the edges of $H$, such that $h_2\cdots h_r-g_2\cdots g_rt^b \in I(Y)\subset I(X)$
and hence
\begin{equation}\label{eq: inequality6}
t^a-h_1g_2\cdots g_rt^b \in I(X).
\end{equation}
If $a_2\not = 0$, then from (\ref{eq: inequality5}) we deduce that
$g_1$ divides $h_1$ and we are done.
If $a_1\not = 0$, then there exists $c\in \NN^s$ such that $h_1=f_1t^c$. Accordingly,
$h_1g_2\cdots g_rt^b = f_1g_2\cdots g_r t^{b+c}$. Since $f_1g_2-f_2g_1 \in I(X)$, we deduce that
$f_1g_2\cdots g_rt^{b+c} - f_2g_1g_3\cdots g_r t^{b+c} \in I(X)$, which, together with (\ref{eq: inequality6}),
implies that
\begin{equation}\label{eq: membership1}
t^a - f_2g_1g_3\cdots g_r t^{b+c} \in I(X).
\end{equation}
Consider $a'\in \NN^s$ such that $t^{a'}=f_2g_1g_3\cdots g_rt^{b+c}$. Since $h_1=f_1t^c$ and
the monomials $f_2, g_3,\dots, g_r, t^b$ are supported away from the edges of $P_1$, we see that, if $1\leq i\leq k_1$,
$a'_i = a_i-1$, when $i$ is odd, and $a'_i=a_i+1$, when $i$ is even. In particular, $a'_2\not =0$
and, in the corresponding decomposition $t^{a'}=h'_1\cdots h'_r$ with
monomials $h'_i \in K[E_{P_i}]$, we get $\deg(h'_1) = \deg(h_1)-1$.
Repeating the previous argument, we deduce that $t^{a'}\in (I(X),g)$,
which, using (\ref{eq: membership1}) implies that $t^{a}\in (I(X),g)$.
\smallskip

We are left with the case of $a_1=a_2=0$. We regard $t^a$ as a monomial in
$K[E_{G'}]$, where $G'$ is the graph obtained as the parallel composition of
$P_1\setminus \set{t_1,t_2}$, $P_2,\dots, P_r$.
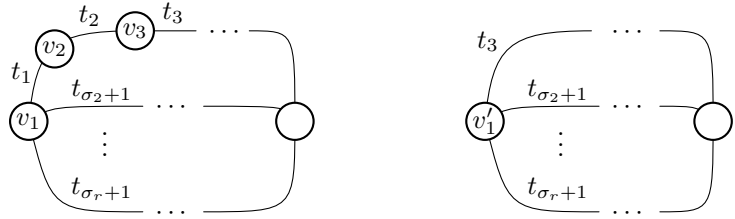
\begin{figure}[h]
\begin{center}
\begin{tikzpicture}[line cap=round,line join=round, scale=2]

\draw (0,-.6)..controls (0.05,-.1) and (0.15,0)..(0.7,0);
\draw (0.804,0)--(1.1,0);
\draw (0,-.6).. controls (0.15,-.5)..(.75,-.5);
\draw (0.95,-.505) node{$\cdots$};

\draw (1.45,0)..controls (1.75,0)..(1.75,-.6);
\draw (1.3,-.006) node{$\cdots$};

\draw (0,-.6)..controls (0.15,-1.2)..(.75,-1.2);
\draw (0.95,-1.206) node{$\cdots$};

\draw [fill=white,thick] (0,-.6) circle (3.5pt);
\draw (0,-.6) node {$v_1$};
\draw (-.045,-.28) node {$t_{1}$};

\draw [fill=white,thick] (0.17,-.12) circle (3.5pt);
\draw (0.17,-.12) node {$v_2$};
\draw (0.4,.1) node {$t_{2}$};

\draw [fill=white,thick] (0.7,0) circle (3.5pt);
\draw (0.7,0) node {$v_3$};
\draw (0.95,.12) node {$t_{3}$};

\draw (1.15,-.5)..controls (1.7,-.5)..(1.75,-.6);
\draw (.48,.-.4) node {$t_{\sigma_2+1}$};

\draw (1.15,-1.2)..controls (1.75,-1.2)..(1.75,-.6);
\draw (.48,.-1.05) node {$t_{\sigma_r+1}$};
\draw (0.5,-0.7) node {$\vdots$};

\draw [fill=white,thick] (1.75,-.6) circle (3.5pt);


\draw (3,-.6)..controls (3.05,-.1) and (3.15,0)..(3.7,0);
\draw (3,-.6)..controls (3.15,-.5)..(3.75,-.5);
\draw (3.95,-.505) node{$\cdots$};

\draw (4.15,0)..controls (4.5,0)..(4.5,-.6);
\draw (3.95,-.006) node{$\cdots$};

\draw (3,-.6)..controls (3.15,-1.2)..(3.75,-1.2);
\draw (3.95,-1.206) node{$\cdots$};

\draw [fill=white,thick] (3,-.6) circle (3.5pt);
\draw (3,-.6) node {$v'_1$};
\draw (3.02,-.07) node {$t_{3}$};

\draw (4.15,-.5)..controls (4.4,-.5)..(4.5,-.6);
\draw (3.48,.-.4) node {$t_{\sigma_2+1}$};

\draw (4.15,-1.2)..controls (4.5,-1.2)..(4.5,-.6);
\draw (3.48,.-1.05) node {$t_{\sigma_r+1}$};
\draw (3.5,-0.7) node {$\vdots$};

\draw [fill=white,thick] (4.5,-.6) circle (3.5pt);

\end{tikzpicture}
\end{center}
\caption{$G$ (left) and $G'$ (right).}
\label{fig: graphs G and G'}
\end{figure}

\noindent
Let $X'$ be the projective toric subset
parameterized by $G'$. By the induction hypothesis
$$
\reg K[E_{G'}]/I(X') = (\lfloor k_1/2 \rfloor+\cdots+\lfloor k_r/2 \rfloor -1)(q-2).
$$
Hence, by Proposition~\ref{prop: computing reg by reducing to Artinian quotient}
and Proposition~\ref{prop: membership of a zero dimensional reduction}, there exists
$t^d\in K[E_{G'}]$, where $d\in \NN^s$ is supported on the edges of $G'$, such that
\begin{equation}
\label{eq: binomial in G'}
t^a-g'_1g_2\cdots g_rt_3^{q-1}t^d \in I(X'),
\end{equation}
where $g'_1=g_1/t_2\in K[E_{G'}]$.
We claim there exists $k\in \set{1,2,\dots,q-1}$ such that
\begin{equation}
\label{eq: binomial in G}
t^a -t_1^{\widehat{k}}t_2^k g'_1g_2\cdots g_rt^d \in I(X)
\end{equation}
with $\widehat{k}=q-1-k$. We define $k$ using the congruence at vertex $v'_1$ of $G'$
(see Figure~\ref{fig: graphs G and G'}) which, according to \cite[Lemma~2.3]{NeVP14},
is satisfied for the binomial in (\ref{eq: binomial in G'}). This congruence is:
\begin{equation}\nonumber
\renewcommand{\arraystretch}{1.3}
\begin{array}{c}
a_3+a_{\sigma_2+1}+\cdots + a_{\sigma_r+1}\equiv q-1 + d_3 + d_{\sigma_2+1} +\cdots + d_{\sigma_r+1} \pmod{q-1} \\
\iff a_3-d_3 \equiv d_{\sigma_2+1}+\cdots +d_{\sigma_r+1}  - (a_{\sigma_2+1} +\cdots + a_{\sigma_r+1} )  \pmod{q-1}.
\end{array}
\end{equation}
Let $k\in \set{1,2,\dots,q-1}$ to be such that:
\begin{equation}\label{eq: def of k}
k \equiv a_3-d_3 \equiv d_{\sigma_2+1}+\cdots +d_{\sigma_r+1}  - (a_{\sigma_2+1} +\cdots +  a_{\sigma_r+1} )  \pmod{q-1}.
\end{equation}
Let us now show that (\ref{eq: binomial in G}) holds. Since $t^a -t_1^{\widehat{k}}t_2^k g'_1g_2\cdots g_rt^d$ is homogeneous,
it will suffice to check the congruences at each vertex of $G$. Since for the binomial in (\ref{eq: binomial in G'}), from which
we obtain this binomial, the congruences are satisfied at all vertices of $G'$,
it will be enough to check the congruences for the vertices
$v_1,v_2$ and $v_3$. At $v_1$, we have:
$$
a_{\sigma_2+1}+\cdots + a_{\sigma_r+1}\equiv (q-1)-k + d_{\sigma_2+1}+\cdots + d_{\sigma_r+1} \pmod{q-1},
$$
at $v_2$,
$0\equiv (q-1)-k+k \pmod{q-1}$
and at $v_3$,
$a_3 \equiv k + d_{3} \pmod{q-1}$,
all of which hold, by virtue of (\ref{eq: def of k}). This completes the proof of the theorem.
\end{proof}

The proof of Theorem~\ref{th: main theorem} follows from Lemma~\ref{lemma: easy inequality},
Proposition~\ref{prop: opp inequality in the case k_i even} and Theorem~\ref{thm: opp inequality in the case k_i odd}.
\smallskip

We now turn to the proof of Theorem~\ref{th: main theorem 2}. In this case, $G$ is the parallel composition of paths $P_1,\dots,P_r$ the lengths of which have mixed parity. 
We assume, without loss of generality, that $P_1,\dots,P_l$ have odd lengths and $P_{l+1},\dots,P_r$ have even lengths, for some $1\leq l <r$. 
We will keep the notation for the edges of $G$ as in the beginning of this section
and recall that (as in the statement of Theorem~\ref{th: main theorem 2}) we will be denoting by $H_1$ the parallel composition of the paths of odd lengths, 
by $H_2$ the parallel composition of the paths of even lengths and by $X_1,X_2$, 
respectively, the projective toric subsets they parameterize. 

\begin{lemma}\label{lemma: easy inequality in the non-bipartite case}
$\reg K[E_G]/I(X) \leq  \reg K[E_{H_1}]/I(X_1) + \reg K[E_{H_2}]/I(X_2) + (q-2)$.
\end{lemma}

\begin{proof}
Consider the cover of $G$ given by $H_1$ and $H'_2$ where 
$H'_2$ is given by $\set{t_1}\cup H_2$. Then \mbox{$E_{H_1}\cap E_{H'_2} \not = \emptyset$} and therefore by 
Proposition~\ref{prop: regularity inequality over subgraphs},
\begin{equation}\label{eq: proof of easy inequality in the non-bipartite case}
\reg K[E_G]/I(X) \leq \reg K[E_{H_1}]/I(X_1) + \reg K[E_{H'_2}]/I(X'_2),
\end{equation}
where $X'_2$ is the projective toric subset parameterized by $H'_2$.
By Proposition~\ref{prop: aditivity on leaves}, we know that $\reg K[E_{H'_2}]/I(X'_2)= \reg K[E_{H_2}]/I(X_2) + (q-2)$. Combining
this with (\ref{eq: proof of easy inequality in the non-bipartite case}) completes the proof of the lemma.
\end{proof}

\begin{theorem}\label{thm: harder inequality in the non-bipartite case}
$\reg K[E_G]/I(X) \geq  \reg K[E_{H_1}]/I(X_1) + \reg K[E_{H_2}]/I(X_2) + (q-2)$.
\end{theorem}

\begin{proof}
We divide the proof into cases. 
\smallskip

\noindent
\emph{The case $l=1$ and $r=2$.} In this case $G$ is a cycle 
of (odd) length $k_1+k_2$. Accordingly, $X$ coincides with $\TT^{k_1+k_2-1}$ and, by the formula in Table~\ref{table: values of reg},
\mbox{$\reg K[E_G]/I(X) = (k_1+k_2-1)(q-2)$.} On the other hand $H_1$ and $H_2$ are paths of lengths $k_1$ and $k_2$ and the projective 
toric subsets they parameterized are the tori $\TT^{k_1-1}$ and $\TT^{k_2-1}$ so that, again by the same formula,
$\reg K[E_{H_1}]/I(X_1)=(k_1-1)(q-2)$ and $\reg K[E_{H_2}]/I(X_2) = (k_2-1)(q-2)$. We deduce that
$$
\reg K[E_G]/I(X) =  \reg K[E_{H_1}]/I(X_1) + \reg K[E_{H_2}]/I(X_2) + (q-2).
$$

\noindent
In the other cases, we will use vertex identifications and Proposition~\ref{prop: general lower bound in contraction}. 
For this purpose, let us denote the terminal vertices of the parallel composition yielding $G$ by $v$ and $w$.
\smallskip

\noindent
\emph{The case $l=1$, $k_1=1$ and $r-l>1$.}
Consider the vertices of $P_2,\dots,P_r$ at an odd number of edges away from $v$ (or $w$). They form 
an independent set of vertices of cardinality $k_2/2+\cdots +k_r/2$. Then, by Proposition~\ref{prop: lower bound using independent set},
we get 
\begin{equation}\label{eq: inequality when there is only one odd path}
\reg K[E_G]/I(X) \geq (k_2/2+\cdots+k_r/2)(q-2).
\end{equation}
Now, by Theorem~\ref{th: main theorem}, the right-hand of (\ref{eq: inequality when there is only one odd path}) is equal to 
$$0+\reg K[E_{H_2}]/I(X_2) + (q-2) = \reg K[E_{H_1}]/I(X_1) + \reg K[E_{H_2}]/I(X_2) + (q-2).$$ 
\smallskip

\noindent
\emph{The case $l=1$, $k_1>1$ and $r-l>1$.}
Let $G'$ be the graph obtained by identifying all the vertices in the paths $P_2,\dots,P_r$ 
at an even number of edges away from $v$ 
(or $w$) with the vertex $v$.
The resulting graph $G'$ consists of an odd cycle of length $k_1$ with a set of 
$k_2/2+\cdots +k_r/2$ double 
edges incident to one of its vertices (cf.~Figure~\ref{fig: First reduction}).
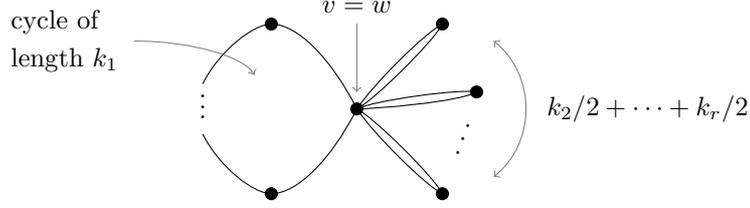
\begin{figure}[h]
\begin{center}
\begin{tikzpicture}[line cap=round,line join=round, scale=2.25]

\draw [fill=black] (0,0) circle (1pt);
\draw [black] (0,0).. controls (-.15,.3) and (-.35,.5) ..(-.5,.5);
\draw [fill=black] (-.5,.5) circle (1pt);

\draw [black] (0,0).. controls (-.15,-.3) and (-.35,-.5) ..(-.5,-.5);
\draw [fill=black] (-.5,-.5) circle (1pt);

\draw [black] (-.5,.5).. controls (-.6,.5) and (-.8,.35) ..(-.9,.15);
\draw [black] (-.5,-.5).. controls (-.6,-.5) and (-.8,-.35) ..(-.9,-.15);

\draw (-.9,.06) node {$\vdots$};

\draw [black] (0,0).. controls (.1,.15) and (.4,.45) ..(.5,.5);
\draw [black] (0,0).. controls (.1,.05) and (.45,.4) ..(.5,.5);
\draw [fill=black] (.5,.5) circle (1pt);

\draw [black] (0,0).. controls (.1,.05) and (.6,.12) ..(.7,.1);
\draw [black] (0,0).. controls (.1,0) and (.6,.04) ..(.7,.1);
\draw [fill=black] (.7,.1) circle (1pt);

\draw (.65,-.10) node {$\cdot$};
\draw (.62,-.18) node {$\cdot$};
\draw (.59,-.26) node {$\cdot$};

\draw [black] (0,0).. controls (.1,-.15) and (.4,-.45) ..(.5,-.5);
\draw [black] (0,0).. controls (.1,-.05) and (.45,-.4) ..(.5,-.5);
\draw [fill=black] (.5,-.5) circle (1pt);

\draw[gray,->] (0,.5) -- (0,.1); 
\draw (0,.6) node {$v=w$};

\draw[gray,->] (-1.3,.4).. controls (-.95,.4) and (-.65,.3).. (-.6,.2); 
\draw (-1.7,.4) node {\parbox{1.45cm}{\begin{flushleft}cycle of length $k_1$\end{flushleft}}};

\draw[gray,<->] (.8,.4).. controls (1.05,.2) and (1.05,-.2).. (.8,-.4); 
\draw (1.7,0) node {$k_2/2+\cdots + k_r/2$};


\end{tikzpicture}
\end{center}
\caption{$G'$, obtained by identifying every other vertex in $P_2,\dots,P_r$.}
\label{fig: First reduction}
\end{figure}
Let $X'$ the projective toric subset parameterized by $G'$. 
The regularity of $K[E_{G'}]/I(X')$ is the same as if in $G'$ all double 
edges were single edges. Hence by Proposition~\ref{prop: aditivity on leaves} and the formula for 
the odd cycle case we get:
$$
K[E_{G'}]/I(X') = (k_1-1)(q-2) + (k_2/2+\cdots + k_r/2)(q-2)
$$
which coincides with $\reg K[E_{H_1}]/I(X_1) + \reg K[E_{H_2}]/I(X_2) + (q-2)$. Since, 
by Proposition~\ref{prop: general lower bound in contraction}, $\reg K[E_G]/I(X)\geq \reg[E_{G'}]/I(X')$
we obtain the desired inequality.
\smallskip

\noindent
\emph{The case $l>1$ and $r-l=1$.} In this case we construct a graph $G'$ by identifying  
all vertices in $P_1,\dots,P_l$ at an even number of edges away from $v$ with the vertex $v$. 
\begin{figure}[h]
\begin{center}
\begin{tikzpicture}[line cap=round,line join=round, scale=2.25]

\draw[fill=black] (0,0) circle (1pt);
\draw (0,0)..controls (.1,.15	) and (.9,.15) .. (1,0);
\draw (.5,.05) node {$\vdots$};
\draw (0,0)..controls (.1,-.15) and (.9,-.15) .. (1,0);
\draw[fill=black] (1,0) circle (1pt);

\draw (1.1,.1) node {$w$};
\draw (-.075,-.125) node  {$v$};

\draw (0,0)..controls (.1,-.6) and (.9,-.6) .. (1,0);
\draw[fill=black] (.5,-.45) circle (1pt);
\draw[fill=white,white] (.15,-.3) circle (4pt);
\draw (.1,-.25) node {$\cdot$};
\draw (.16,-.3105) node {$\cdot$};
\draw (.22,-.37) node {$\cdot$};
\draw (.8,-.5) node {$P_r$};

\draw (0,0)..controls (-.15,.04) and (-.55,.04) .. (-.7,0);
\draw (0,0)..controls (-.15,-.04) and (-.55,-.04) .. (-.7,0);
\draw[fill=black] (-.7,0) circle (1pt);

\draw (0,0)..controls (-.15,.15) and (-.55,.33) .. (-.62,.35);
\draw (0,0)..controls (-.15,.04) and (-.55,.26) .. (-.62,.35);
\draw[fill=black] (-.62,.35) circle (1pt);

\draw (0,0)..controls (-.04,.1) and (-.04,.6) .. (0,.7);
\draw (0,0)..controls (.04,.1) and (.04,.6) .. (0,.7);
\draw[fill=black] (0,.7) circle (1pt);

\draw (-.43,.50) node {$\cdot$};
\draw (-.33,.55) node {$\cdot$};
\draw (-.23,.60) node {$\cdot$};

\draw[gray,<-] (.65,0).. controls (.5,.3) and (.65,.4).. (.85,.4); 
\draw (1.45,.4) node {\parbox{2.5cm}{\begin{flushleft}$l$ multiple edges\end{flushleft}}};

\draw[gray,<-] (.75,-.25).. controls (.8,-.4) and (1,-.4).. (1.35,-.3); 
\draw (1.9,-.3) node {\parbox{2cm}{\begin{flushleft}cycle of length $k_r+1$\end{flushleft}}};

\draw[gray,<->] (-.8,.175).. controls (-.83,.5) and (-.5,.83).. (-0.175,.8); 
\draw (-1.45,.7) node {$\lfloor k_1/2 \rfloor + \cdots + \lfloor k_{r-1}/2 \rfloor$};


\end{tikzpicture}
\end{center}
\caption{$G'$, obtained by identifying with $v$ every other vertex in $P_1,\dots,P_{r-1}$.}
\label{fig: Second reduction}
\end{figure}
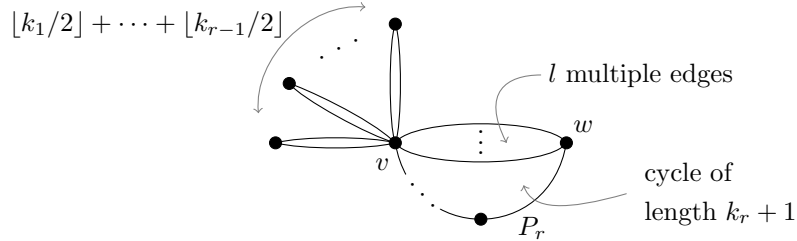
This graph consists of an odd cycle of length $k_r+1$ (given
by $P_r$ and (a choice of) an edge $\set{v,w}$) that has $l$ multiple edges between $v$ and $w$ and of a set of 
$\lfloor k_1/2 \rfloor + \cdots \lfloor k_{r-1}/2 \rfloor$ double edges incident at $v$ (cf.~Figure~\ref{fig: Second reduction}). 
Arguing as above, we get:
$$
\renewcommand{\arraystretch}{1.3}
\begin{array}{c}
\reg K[E_G]/I(X) \geq \reg K[E_{G'}]/I(X') = (\lfloor k_1/2 \rfloor + \cdots \lfloor k_{r-1}/2 \rfloor)(q-2) + k_r(q-2)\\
= \reg K[E_{H_1}]/I(X_1) + \reg K[E_{H_2}]/I(X_2) + (q-2).
\end{array}
$$

\noindent
\emph{The case $l>1$ and $r-l>1$.} As in the previous case, let $G'$ be the graph obtained by identifying  
the vertices in $P_1,\dots,P_l$ at an even number of edges away from $v$ with this vertex. 
We notice that the subgraph of $G'$ consisting of the paths $P_{l+1},\dots,P_r$ and (a choice of) an edge $\set{v,w}$ belongs
to the second case, above. Consequently,
$$
\renewcommand{\arraystretch}{1.3}
\begin{array}{c}
\reg K[E_G]/I(X) \geq (\lfloor k_1/2 \rfloor + \cdots + \lfloor k_l/2 \rfloor)(q-2) + (k_{l+1}/2+\cdots+k_r/2)(q-2)\\
= \reg K[E_{H_1}]/I(X_1) + \reg K[E_{H_2}]/I(X_2) + (q-2). \qedhere
\end{array}
$$
\end{proof}

The proof of Theorem~\ref{th: main theorem 2} is obtained by combining Lemma~\ref{lemma: easy inequality in the non-bipartite case} and
Theorem~\ref{thm: harder inequality in the non-bipartite case}. In Table~\ref{table: values of reg for a parallel composition} we
give explicit formulas for the regularity of $K[E_G]/I(X)$ when $G$ is a parallel composition of $r\geq 2$ paths of lengths
$k_1,\dots,k_r$, of which $k_1,\dots,k_l$ are odd and $k_{l+1},\dots,k_r$ are even. 
\begin{table}[h]
\renewcommand{\arraystretch}{1.5}
\begin{tabular}{l|l}
 & $\reg K[E_G]/I(X)$ \\
\hline
$l=0$ & $(k_1/2+\cdots + k_r/2 - 1)(q-2)$ \\
\hline
$l=r$ & $(\lfloor k_1/2 \rfloor + \cdots + \lfloor k_r/2 \rfloor)(q-2)$ \\
\hline
$l=1,r=2$ & $(k_1+k_2-1)(q-2)$ \\
\hline
$l=1, r>2$ & $(k_1+k_2/2+\cdots +k_r/2 -1)(q-2)$ \\
\hline
$l>1, r=l+1$ & $(\lfloor k_1/2 \rfloor + \cdots + \lfloor k_{r-1}/2 \rfloor + k_r)(q-2)$ \\
\hline
$l>1, r>l+1$ & $(\lfloor k_1/2 \rfloor + \cdots + \lfloor k_l/2 \rfloor + k_{l+1}/2+\cdots + k_r/2)(q-2)$ \\
\hline
\end{tabular}
\bigskip
\label{table: values of reg for a parallel composition}
\caption{Values of $\reg K[E_G]/I(X)$ when $G$ is a parallel composition of paths.}
\end{table}

\end{document}